\documentclass[11pt,a4paper,twoside,draft]{article}

\usepackage[T1]{fontenc}
\usepackage{graphicx}
\usepackage{amsmath,amsthm,amsfonts}

\newtheorem{thm}{Theorem}[section]
\newtheorem{prop}[thm]{Proposition}
\newtheorem{lem}[thm]{Lemma}
\newtheorem{cor}[thm]{Corollary}

\newtheorem{rem}[thm]{Remark}

\newtheoremstyle{note}{3pt}{3pt}{\rm}{}
{\bf}{.}{.5em}{}
\theoremstyle{note}
\newtheorem{ex}[thm]{Example}

\title{Simultaneously Non-dense Orbits\\
 Under Different Expanding Maps}
\author{David  F\"arm\\
\scriptsize{Centre for Mathematical Sciences}\\
\scriptsize{Lund University, Sweden}\\
\scriptsize{david@maths.lth.se}}

\begin{document}

\maketitle

\begin{abstract}
\noindent Given a point and an expanding map on the unit interval, we consider the set of points for which the forward orbit under this map is bounded away from the given point. For maps like multiplication by an integer modulo 1, such sets have full Hausdorff dimension. We prove that such sets have a large intersection property, \mbox{i.e.} that countable intersections of such sets also have full Hausdorff dimension. This result applies to maps like multiplication by integers modulo 1, but also to nonlinear maps like $x \mapsto 1/x$ modulo 1. We prove that the same thing holds for multiplication modulo 1 by a dense set of non-integer numbers between 1 and 2.
\end{abstract}

\section{Introduction}
\subsection{Multiplication by integers modulo 1}
It is well-know that for maps like $f_b \colon [0,1)\to [0,1)$ where $f\colon x \mapsto bx \mod 1$ and $b$ is an integer larger than one, the forward orbit $(f^n(x))_{n=0}^\infty$ is dense for almost all points with respect to the Lebesgue measure. It follows that sets like
\[
G_{f_b}(x) :=\Big\{\, y\in [0,1): x\notin \overline{\cup_{n=0}^\infty f_b^n(y)}\, \Big\},
\]
where $x\in [0,1]$, have zero measure. On the other hand, it is not difficult to see that such sets have full Hausdorff dimension. In this paper we will consider what happens if we start intersecting such sets. For example we will prove a theorem that implies 
\[
\dim_H(G_{f_2}(x) \cap G_{f_3}(x))=1
\]
and even 
\[
\dim_H \Big(\bigcap_{b=2}^\infty G_{f_b}(x_b) \Big)=1,
\]
where $x_b\in [0,1]$ for all $b$. The key property of $f_b$ is that it generates a symbolic representation of $[0,1)$. Indeed, any number $x\in [0,1)$ can be represented as a sequence $(x_i)_{i=1}^\infty\in \{0,1,\dots b-1\}^{\mathbb N}$, where $x=\sum_{i=1}^\infty \frac{x_i}{b^i}$. This representation is unique except on a countable set. Since we are only interested in Hausdorff dimension this ambiguity can be disregarded. Now, we have a correspondence between $[0,1)$ and $\Sigma_b:=\{0,1,\dots b-1\}^{\mathbb N}$ where $f_b\colon [0,1)\to [0,1)$ corresponds to the left shift $\sigma \colon \Sigma_b \to \Sigma_b$, where $\sigma \colon (x_i)_{i=1}^\infty \mapsto (x_{i+1})_{i=1}^\infty$. Now, instead of considering the set $G_{f_b}(x)$ directly, we can consider the set
\[
\cup_{n=1}^\infty \Big\{(y_i)_{i=1}^\infty\in \Sigma_b: x_1\dots x_n \neq y_k \dots y_{k+n-1} \forall k\geq 1\Big\}.
\]
We can handle much more general maps than these, but to state the main theorems we need to define the main tool of this paper, the $(\alpha,\beta)$-game.

\subsection{The $\boldsymbol{(\alpha, \beta)}$-game}
    We will use a one dimensional version of a set theoretic game that was introduced by W. Schmidt in \cite{Schmidt1}. In our case, the game is played on the unit interval $[0,1]$ equipped with Euclidean metric. There are two players, Black and White, and two fixed numbers $\alpha,\beta \in (0,1)$. The rules are as follows.
    \begin{itemize}
      \item {\em In the initial step} Black chooses any closed interval $B_0$, and then White chooses a closed interval $W_0 \subset B_0$ such that $|W_0| = \alpha |B_0|$.
     \item

      {\em Then the following step is repeated.} At step $k$ Black choses a closed interval $B_{k} \subset W_{k-1}$ such that $|B_{k}| = \beta |W_{k-1}|$. Then White chooses a closed interval $W_{k} \subset B_k$ such that $|W_k| = \alpha |B_k|$.
    \end{itemize}

It is clear that the set 
\[
      \bigcap_{k=0}^\infty W_k = \bigcap_{k=0}^\infty B_k
    \]
will always consist of exactly one point. A set $E$ is said to be $(\alpha, \beta)$-winning if White always can achieve that
    \[
      \bigcap_{k=0}^\infty W_k  \subset E.
    \]
    A set $E$ is said to be $\alpha$-winning if it is $(\alpha, \beta)$-winning for all $\beta$. 
    
For us, the key property of $\alpha$-winning sets proved by Schmidt \cite{Schmidt1} can be summarised as follows.
\begin{prop}\label{winninggivesdim}
If the set $E\subset [0,1]$ is $\alpha$-winning for some $\alpha>0$, then $\dim_H(E)=1$.
\end{prop}

\begin{prop}\label{intersectionproperty}
Let $\alpha>0$ and let $(E_i)_{i=1}^\infty$ be a sequence of $\alpha$-winning sets. Then the set $\cap_{i=1}^\infty E_i$ is also $\alpha$-winning.
\end{prop}

\subsection{Expanding maps generating full shifts} \label{givingfullshift}

Let $f\colon [0,1)\to [0,1)$ be such that there are finitely or countably many disjoint intervals $[a,b)\subset [0,1)$ such that $\sum |[a,b)|=1$ and $f|_{[a,b)}$ is monotone and onto for each of these intervals. Note that we do not assume that $f$ is well defined on $[0,1)$, only on each of the intervals $[a,b)$. 

We take an enumeration of the intervals and associate each interval to the corresponding number so that we can refer to an interval as $[n]$ where $n$ is the appropriate number. Assume that for each of the intervals $[a,b)$ it holds that  $|f(x)-f(y)|\geq |x-y|$ for all $x,y \in [a,b)$. Then we can define cylinders 
\[
C_{\, x_1 \dots x_n}:=\Big\{x\in [0,1):\bigcap_{i=1}^n f^{-(i-1)}(x)\in [x_i] \, \Big\}.
\]
If $\lim_{n\to \infty}|C_{x_1 \dots x_n}|= 0$ for all $(x_i)_{i=1}^\infty \in \Sigma:=\{0,1,\dots b-1\}^{\mathbb N}$ or $\{0,1,\dots \}^{\mathbb N}$ depending on if $[0,1)$ was split into finitely or infinitely many parts, we can represent $[0,1)$ by $\Sigma$. If the alphabet is infinite, some points in $[0,1)$ may not have a well-defined expansion. For example, with $f\colon  x \mapsto \frac{1}{x} \mod 1$ we cannot represent the set $
\bigcup_{n=0}^\infty f^{-n}(\{0\})$ in $\Sigma$. It is clear that at least Lebesgue almost every point has a well defined expansion.

To study sets like
\[
G_{f}(x) :=\Big\{\, y\in [0,1): x\notin \overline{\cup_{n=0}^\infty f^n(y)}\, \Big\},
\]
we will use their representation in $\Sigma$ which in this case is 
\[
\bigcup_{n=1}^\infty \Big\{\, (y_i)_{i=1}^\infty\in \Sigma: x_1\dots x_n \neq y_k \dots y_{k+n-1} \ \forall k\geq 1\, \Big\}.
\]

The key theorem of this paper is the following. We will discuss conditions $(i)$ and $(ii)$ in Section \ref{conditions}. 

\begin{thm}\label{fullshiftinunitinterval}
Let $f$ be as described above and such that it satisfies the following conditions.
\begin{enumerate}
\item[(i)]
There exists an $\alpha_0>0$ such that for each $k \in \mathbb N$, each closed interval ${I\subset [0,1)}$ and each $\beta>0$, when playing the $(\alpha_0,\beta)$ game with ${B_0=I}$, after a finite number of turns White is able put his set $W_j$ in a generation $k$ cylinder for some $j$, thereby avoiding all endpoints of generation $k$ cylinders.

\item[(ii)]
There is a positive function $g\colon  \mathbb N \to [0, \infty)$ such that $g(m)\to 0$ as $m\to \infty$ and 
\[
\frac{|C_{x_1\dots x_{n+m}}|}{|C_{x_1\dots x_n}|}\leq g(m)
\]
for all $(x_i)_{i=1}^\infty\in \Sigma$ and all $n,m \in \mathbb N$. 
\end{enumerate}

\noindent Then for any $x\in [0,1)$ which has a well-defined expansion there is an $\alpha>0$ such that the set 
\[
G_f(x)=\Big\{\, y\in [0,1): x\notin \overline{\cup_{n=0}^\infty f^n(y)}\, \Big\},
\]
is $\alpha$-winning in $[0,1]$. In fact $\alpha=\min\{\alpha_0, \frac{1}{4}\}$ is small enough.
\end{thm}
The main result of the paper is the following corollary which follows after using Proposition \ref{winninggivesdim} and  Proposition \ref{intersectionproperty}.

\begin{cor}
Let $(f_i)_{i=1}^\infty$ be a sequence of functions as in Theorem \ref{fullshiftinunitinterval} and let $(x_i)_{i=1}^\infty$ be a sequence of points in $[0,1)$ with well-defined expansions. Then 
\[
\dim_H \Big(\bigcap_{i=1}^\infty  G_{f_i}(x_i)\Big)=1.
\]
\end{cor}

\subsection{$\boldsymbol{\beta}$-shifts where the expansion of 1 terminates}

The following method to expand real numbers in non-integer bases was introduced by R\'enyi~\cite{Renyi} and Parry~\cite{Parry}. For more details and proofs of the statements below, see their articles. 

Let $[x]$ denote the integer part of the number $x$. Let $\beta \in (1,2)$. For any $x \in [0, 1]$ we associate the sequence $d(x,\beta) =  \{d_n (x, \beta)\}_{n=0}^\infty \in \{0, 1\}^\mathbb N$ defined by
    \[
      d_n (x, \beta) :=
      [\beta f_\beta^{n} (x)],
    \]
    where $f_\beta (x) = \beta x \mod 1$.
    The closure of the set
    \[
      \{\, d (x,\beta) : x \in [0,1)\,\}
    \]
    is denoted by $S_\beta$ and it is called the $\beta$-shift. It is invariant under the left-shift $\sigma \colon \{i_n\}_{n=0}^\infty \mapsto \{i_{n+1} \}_{n=1}^\infty$ and the map $d (\cdot, \beta) \colon x \mapsto d (x, \beta)$ satisfies $\sigma^n ( d (x, \beta) ) = d ( f_\beta^n (x), \beta)$. If we order $S_\beta$ with the lexicographical ordering then the map $d ( \cdot, \beta)$ is  one-to-one and monotone increasing. The subshift $S_\beta$ satisfies
    \begin{equation} \label{eq:Sbeta}
      S_\beta = \{\, \{j_k\} : \sigma^n \{j_k\} < d (1, \beta) \ \forall n \,\}.
    \end{equation}
  
     If $x \in [0,1]$ then
    \[
      x = \sum_{k=0}^\infty \frac{d_k (x, \beta) }{\beta^{k + 1}}.
    \]
    We let $\pi_\beta$ be the map $\pi_\beta \colon S_\beta \to [0,1)$ defined by
    \[
      \pi_\beta \colon \{i_k\}_{k=0}^\infty \quad \mapsto \quad \sum_{k=0}^\infty \frac{i_k}{\beta^{k + 1}}.
    \]
    Hence, $\pi_\beta ( d(x, \beta)) = x$ holds for any $x \in [0,1)$ and $\beta > 1$.

    A cylinder $s$ is a subset of $[0,1)$ such that 
    \[
      s  :=\pi_\beta( \{\, \{j_k\}_{k=0}^\infty : i_k = j_k,\ 0\leq k < n \,\})
    \]
    holds for some $n$ and some sequence $\{i_k\}_{k=0}^\infty$. We then say that $s$ is an $n$-cylinder or a cylinder of generation $n$ and write
    \[
      s = [i_0 \cdots i_{n-1}].
    \]
    Consider $\beta$ such that the expansion of 1 terminates, \mbox{i.e.} such that $d(1, \beta)=j_0 \dots j_{k-1} 0^\infty$. The set of such $\beta$ is dense in $(1,2)$ and for such $\beta$ we can use (\ref{eq:Sbeta}) to construct $S_\beta$ from the full shift $\Sigma_2=\{0,1\}^{\mathbb N}$ as follows. There are finitely many words $w$ of length $k$ such that $w<d(1, \beta)$. If we start with $\Sigma_2$ and remove all elements that contain any of these words, then by  (\ref{eq:Sbeta})  we get $S_\beta$. Thus $S_\beta$ is a subshift of finite type. Such shifts have have well-known properties that we can use to prove the following theorem.

\begin{thm}\label{SFTunitinterval}
Let $\beta \in (1,2)$ be such that the expansion of 1 terminates. Then for any $x\in [0,1]$ there is an $\alpha>0$ such that the set 
\[
G_{f_\beta}(x)=\Big\{\, y\in [0,1): x\notin \overline{\cup_{n=0}^\infty f^n(y)}\, \Big\},
\]
is $\alpha$-winning in $[0,1]$. In fact $\alpha=\frac{1}{4}$ is small enough.
\end{thm}

Using Proposition \ref{winninggivesdim} and  Proposition \ref{intersectionproperty} we get 

\begin{cor}
Let $(\beta_i)_{i=1}^\infty$ be a sequence in $(1,2)$ such that that the expansion of 1 terminates for each $\beta_i$ and let $(x_i)_{i=1}^\infty$ be a sequence of points in $[0,1]$. Then 
\[
\dim_H \Big(\bigcap_{i=1}^\infty  G_{f_{\beta_i}}(x_i)\Big)=1.
\]
\end{cor}

\section{Conditions on the maps}\label{conditions}

\subsection{Condition $\mathbf{(i)}$}
Assume that we did not have condition $(i)$. Depending on $f$, there might be points in $[0,1)$ which do not have well-defined representations as sequences. We will be playing the $(\alpha,\beta)$ game, trying to show that our sets are $\alpha$-winning. But if no further restrictions are put on $f$ this will not be possible, as the following example illustrates. 

\begin{ex} 
We are going to construct a function $f$ such that for each $\alpha>0$ there is a $\beta>0$ for which the set of points with well-defined representations as sequences is not $(\alpha, \beta)$-winning. First divide $(0,1)$ into the intervals $[\frac{1}{2^i},\frac{1}{2^{i-1}})$ where $i\in \mathbb N$. For each $i$ consider the corresponding interval. Split the interval into $4i$ subintervals of equal size. On every second of these let $f$ be linear onto $[0,1)$. Take all of the remaining subintervals and split them into $4i$ parts and continue this procedure indefinitely. After doing this for each $i$ we have defined a function $f$ except on a set of Lebesgue measure zero. Although this set is small with respect to Lebesgue measure we get into trouble. 

For any $\alpha>0$, pick an $i \in \mathbb N$ such that $\frac{1}{i}<\alpha$. Let $\beta$ be such that $\alpha \beta = \frac{1}{4i}$. Let the player Black choose $B_0$ as the interval $[\frac{1}{2^i},\frac{1}{2^{i-1}})$. Then no matter how White chooses $W_0$, it is always possible for Black to choose $B_1$ as one of the $2i$ intervals on which $f$ was not defined until at smaller scale. The player Black can play so that this situation is repeated indefinitely. So, the points at which $f$ is well-defined is not $(\alpha,\beta)$-winning. Since $\alpha>0$ was arbitrary, this set is not $\alpha$-winning for any $\alpha>0$. So, with this $f$, we cannot use the $(\alpha,\beta)$-game.
\end{ex}

It is clear that we avoid cases like this if we impose condition $(i)$ on $f$. For a given function $f$, condition $(i)$ may not be that easy to check so we give a sufficient condition for it to be satisfied.

\begin{lem}\label{endpointscondition}
Let $f$ be a function as described in Section \ref{givingfullshift} and let $E(f)$ be the set of endpoints of generation $1$ cylinders. Let $Acc(E)$ denote the set of points of accumulation for a set $E$. If there is an $n\in \mathbb N$ such that $Acc^n(E(f))=\emptyset$, then condition $(i)$ is satisfied.
\end{lem}

\begin{proof}
Assume that White is given an interval $I$ and wants to avoid all endpoints of generation $k$ cylinders. Since $Acc^n(E(f))$ is empty we know that $Acc^{n-1}(E(f))$ is finite. It is then easy for White to avoid this set in finitely many turns if $\alpha \leq \frac{1}{2}$. When this is done, White has placed a set $W_{j_1}$ such that it does not contain any points from $Acc^{n-1}(E(f))$. But then it can at most contain finitely many points from $Acc^{n-2}(E(f))$. Of course White can avoid these in the same way. By induction, White can avoid all points from $E(f)$ in a finite number of turns. This means that White can choose a set $W_{j_n}$ inside a generation $1$ cylinder $C_{x_1}$ after finitely many turns. Let $E(f^2)$ denote the set of endpoints of generation $2$ cylinders in $C_{x_1}$. If $n\geq 2$, White wants to avoid this set as well. But $E(f^2)$ is the inverse image of $E(f)$ under the homeomorphism $f|_{C_{x_1}}\colon C_{x_1} \mapsto [0,1)$. Thus, $E(f^2)$ has the same topological properties as $E(f)$. In particular, $Acc^n(E(f^2))=\emptyset$, so just as he avoided $E(f)$, White can avoid $E(f^2)$ in finitely many turns if $\alpha \leq \frac{1}{2}$. Repeating this argument, we get that White can avoid all endpoints of generation $k$ cylinders after a finite number of turns and place his set $W_j$ inside a generation $k$ cylinder for some finite $j$.
\end{proof}

Note that while the condition in Lemma \ref{endpointscondition} is sufficient to ensure condition $(i)$ it is by no means necessary. For example consider the middle third Cantor set. It is defined by repeatedly removing the middle third of each interval, starting with $[0,1]$. Let $f$ be the function obtained by letting $f$ be linear from $0$ to $1$ on each removed interval. Then $f$ is well-defined except on the middle third Cantor set which is a perfect set. Thus the conditions of Lemma \ref{endpointscondition} are not fulfilled but it is obvious that in the $(\alpha,\beta)$-game, White only needs one turn to avoid the middle third Cantor set if $\alpha = \frac{1}{9}$. The set of endpoints of cylinders from higher generation will only be scalings of $E(f)$ since $f$ is linear on each cylinder. Thus, White can avoid the endpoints of the cylinders of any given generation in finitely many turns. 

\subsection{Condition $\mathbf{(ii)}$}
To be able to prove Theorem \ref{fullshiftinunitinterval} we need the cylinders to shrink in some uniform way. One way to get this is of course to require uniform expansion, \mbox{i.e} that for some $\lambda>0$ it holds that $|f(x)-f(y)|\geq (1+\lambda)|x-y|$ for all $x,y$ in the same generation $1$ cylinder. We use the weaker Condition $\mbox{(ii)}$ to allow functions like $f\colon x\mapsto \frac{1}{x} \mod 1$.

\begin{lem}
The continued fraction expansion of numbers $x\in [0,1)$ which is given by the map $f\colon  x\mapsto \frac{1}{x} \mod 1$,  satisfies Condition $(ii)$.
\end{lem}

\begin{proof}
Let $x\in (0,1)\setminus f^{-1}(0)$. Then $f^\prime (x)=-\frac{1}{x^2}$ and $|f^\prime (x)|\geq 1$. So, with $x\in (0,1)\setminus (f^{-1}(0)\cup f^{-2}(0))$ we have that if $|f^\prime (x)|\leq \frac{9}{4}$, then 
\[
|f^\prime (x)|\leq \frac{9}{4} \quad \Rightarrow \quad x \geq \frac{2}{3} \quad \Rightarrow \quad f(x) \leq \frac{1}{2} \quad \Rightarrow \quad f^\prime (f(x))\geq 4.
\]
So $|(f^2)^\prime (x)| \geq \frac{9}{4}> 2$ for all  $x\in (0,1)\setminus (f^{-1}(0)\cup f^{-2}(0))$. This implies
\[
\frac{|C_{x_1\dots x_{n+m}}|}{|C_{x_1\dots x_n}|} \leq \sup_{x \in \tilde C_{x_1\dots x_{n+m}}} \frac{1}{(f^m)^\prime (x)} \leq 2^{-\lfloor \frac{m}{2}\rfloor} =g(m)
\]
where $\tilde C_{x_1\dots x_{n+m}}$ means the interior of the cylinder $C_{x_1\dots x_{n+m}}$ and $\lfloor \frac{m}{2}\rfloor$ means the integer part of $\frac{m}{2}$. 
\end{proof}

\section{Proofs}
The idea we use to prove Theorem \ref{fullshiftinunitinterval} and Theorem \ref{SFTunitinterval} is to translate the $(\alpha,\beta)$-game into a game where the players are choosing symbols in a sequence rather than choosing intervals. By using a simple combinatorial argument we can then conclude that our sets are $\alpha$-winning.

\subsection{A game of sequence building}\label{sequencebuildingsection}
Consider the following game for two players $\tilde B$ and $\tilde W$ with two parameters $c$ and $n$. The players are building a one sided infinite sequence $y=(y_i)_{i=1}^\infty$ in a finite or countable alphabet. First $\tilde B$ chooses $y=(y_i)_{i=1}^{b_0}$, where he can choose $b_0$ as large as he likes. Then, $(y_i)_{i=b_0+1}^\infty$ is divided into blocks of $n$ symbols. 

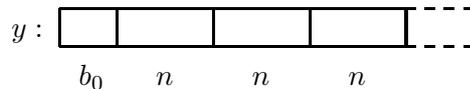
\begin{figure}[ht!]
\begin{center}
\setlength{\unitlength}{0.1\textwidth}
\begin{picture}(6,0.4)(0,0.4)
\thicklines
\put(0,0.4){\line(1,0){3.6}}
\put(0,0.8){\line(1,0){3.6}}
\multiput(3.6,0.4)(0.2,0){4}{\line(1,0){0.1}}
\multiput(3.6,0.8)(0.2,0){4}{\line(1,0){0.1}}

\put(0,0.4){\line(0,1){0.4}}
\put(0.6,0.4){\line(0,1){0.4}}
\put(1.6,0.4){\line(0,1){0.4}}
\put(2.6,0.4){\line(0,1){0.4}}
\put(3.6,0.4){\line(0,1){0.4}}

\put(-0.5,0.5){$y:$}
\put(0.2,0){$b_0$}
\put(1,0){$n$}
\put(2,0){$n$}
\put(3,0){$n$}
\end{picture}
\end{center}
\end{figure}

The game is carried out in one block at a time, so we start in the first block. Consider a list of all possible words of length $n$. This might be infinite depending on whether or not the alphabet is finite. The player $\tilde B$ chooses two disjoint subsets of this list and lets $\tilde W$ pick any one of these two. After $\tilde W$ has made his choice, we have a new list of remaining words. Then $\tilde B$ chooses two disjoint subsets of this list and $\tilde W$ chooses one of these. The players continue like this and the game requires that $\tilde B$ plays so that regardless of how $\tilde W$ plays, this process ends after a finite number of turns, \mbox{i.e.}, that sooner or later only one word remains. This word is then put as $(y_i)_{i=b_0+1}^{b_0+n}$. The same procedure is carried out in each block and we get the sequence $y=(y_i)_{i=1}^\infty$. The game requires that $\tilde W$ gets to play at least $cn$ times in each block regardless of how he plays. This puts restrictions on how $\tilde B$ can construct his subsets. For example, at the first turn in a block, $\tilde B$ cannot choose one of his two subsets to consist of only one word.

\begin{prop}\label{digitgame}
Given any sequence $x=(x_i)_{i=1}^\infty$ and any $c>0$, there is a block size $n$ such that no matter how $\tilde B$ plays in the sequence building game, $\tilde W$ can make sure that there is a number $N$ such that $(x_i)_{i=1}^N \neq (y_i)_{i=k}^{k+N-1}$ for all $k \in \mathbb N$.
\end{prop}

\begin{proof}
Assume that $\tilde B$ chooses the symbols $(y_i)_{i=1}^{b_0}$ and consider $(x_i)_{i=1}^{b_0+2n}$. If we want $y=(y_i)_{i=1}^\infty$ to be such that $(x_i)_{i=1}^{b_0+2n}$ does not occur anywhere in $y$, then it is enough to make sure that none of the $n$-blocks in $y$ occur in $(x_i)_{i=b_0+1}^{b_0+2n}$. 

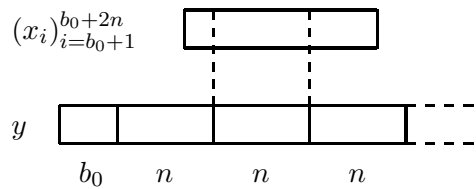
\begin{figure}[ht!]
\begin{center}
\setlength{\unitlength}{0.1\textwidth}
\begin{picture}(6,1.8)(0,0)
\thicklines

\put(1.3,1.4){\line(1,0){2}}
\put(1.3,1.8){\line(1,0){2}}

\put(1.3,1.4){\line(0,1){0.4}}
\put(3.3,1.4){\line(0,1){0.4}}

\put(-0.5,1.5){$(x_i)_{i=b_0+1}^{b_0+2n}$}

\multiput(1.6,0.8)(0,0.18){6}{\line(0,1){0.09}}
\multiput(2.6,0.8)(0,0.18){6}{\line(0,1){0.09}}

\put(0,0.4){\line(1,0){3.6}}
\put(0,0.8){\line(1,0){3.6}}
\multiput(3.6,0.4)(0.2,0){4}{\line(1,0){0.1}}
\multiput(3.6,0.8)(0.2,0){4}{\line(1,0){0.1}}

\put(0,0.4){\line(0,1){0.4}}
\put(0.6,0.4){\line(0,1){0.4}}
\put(1.6,0.4){\line(0,1){0.4}}
\put(2.6,0.4){\line(0,1){0.4}}
\put(3.6,0.4){\line(0,1){0.4}}

\put(-0.5,0.5){$y$}
\put(0.2,0){$b_0$}
\put(1,0){$n$}
\put(2,0){$n$}
\put(3,0){$n$}
\end{picture}
\end{center}
\end{figure}

There are at most $n+1$ different words of length $n$ in $(x_i)_{i=b_0+1}^{b_0+2n}$ and it is sufficient for $\tilde W$ to avoid all these in each $n$-block. We will refer to the words that we want to avoid as dangerous words. In each $n$-block $\tilde W$ gets to make at least $cn$ choices between disjoint collections of words and thereby he can avoid many of the dangerous words. Indeed, the first time $\tilde W$ plays in a block he considers the two disjoint lists of words he is given by $\tilde B$. Since they are disjoint, at least one of the lists contains half or less of the dangerous words. By choosing this list, $\tilde W$ has avoided at least half of the dangerous words in just one play. The next time $\tilde W$ plays he is given two new disjoint lists of words to choose between. Remember that only at most half of the dangerous words are left among these, so $\tilde W$ can avoid at least half of the remaining dangerous words, leaving only at most $\frac{1}{4}$ of the original dangerous words after his second play. Continuing like this, if $2^{cn}>n+1$ he can avoid all the dangerous words in the $cn$ turns he has at each block. Since $c$ is fixed we can always find large enough $n$ such that this is true. It follows that if $\tilde W$ plays according to this strategy we have that $(x_i)_{i=b_0+1}^{b_0+2n} \neq (y_i)_{i=k}^{k+2n-1}$ for any $k\geq b_0+1$, so $(x_i)_{i=1}^{b_0+2n} \neq (y_i)_{i=k}^{k+b_0+2n-1}$ for any $k\geq 1$. Thus, $N=b_0+2n$ will do the job.
\end{proof}

\subsection{Proof of Theorem \ref{fullshiftinunitinterval}} 

The idea of this proof is to create a strategy for White in the $(\alpha,\beta)$-game so that White can play the role of $\tilde W$ in the sequence building game of Section \ref{sequencebuildingsection}. We can then use Proposition \ref{digitgame} to finish the proof. It might take several turns by White to be able to do what $\tilde W$ is supposed to do in one play. Each turn by $\tilde W$ will be divided into two phases consisting of turns by White. In the first phase, the task is to choose between disjoint collections of cylinders of some generation $k_i+k$. In the next phase, the task is to make sure that the game continues inside only one cylinder of generation $k_i+k$. This is to make sure that when we start over with phase one, the cylinders we are choosing between, all have the same coding up to the position $k_i+k$. Then choosing between disjoint collections of cylinders of generation $k_{i+1}+k$ is in fact the same thing as choosing between disjoint collections of codings of positions $k_i+k+1, \dots ,k_{i+1}+k$.

\vspace{0.3 cm}

\noindent The $(\alpha,\beta)$-game starts when the player Black chooses his interval $B_0\subset [0,1]$. 

\vspace{0.3 cm}

\noindent $\mathbf{Phase 1:}$ Let $k_0$ be the largest generation for which there is a cylinder $C_{x_1 \dots x_{k_0}}$ intersecting $B_0$ such that $|C_{x_1 \dots x_{k_0}}|\geq |B_0|$. It might for example be that $k_0=0$, so that $C_{x_1 \dots x_{k_0}}=[0,1)$. By the maximality of $k_0$, all generation $k_0+1$ cylinders intersecting $B_0$ are smaller than $|B_0|$. By condition $(i)$ we know that all cylinders of generation $k_0+k$ intersecting $B_0$ are smaller than $g(k-1)|B_0|$. Let $k$ be a number such that $g(k-1)< \frac{1}{4}$. This is possible since $g(n)\to 0$ as $n\to \infty$, and it implies that the largest cylinder of generation $k_0+k$ intersecting $B_0$ is smaller than $\frac{|B_0|}{4}$.

Let $C'$ be the generation $k_0+k$ cylinder containing the center point of $B_0$. It follows that $B_0 \setminus C'$ consists of two intervals, each of length larger than $\frac{|B_0|}{4}$. Each of these intervals intersects a family of generation $k_0+k$ cylinders and these two families are disjoint. Each family of generation $k_0+k$ cylinders corresponds to a family of codings of positions $1,\dots,k_0+k$. Recall that in the $(\alpha,\beta)$-game, after Black chooses $B_0$, the other player, White, chooses a ball $W_0\subset B_0$ such that $|W_0|=\alpha|B_0|$. So, with $\alpha \leq \frac{1}{4}$, White can choose between two disjoint collections of codings at positions $1, \dots, k_0+k$ by placing $W_0$ to the left or right of $C'$.

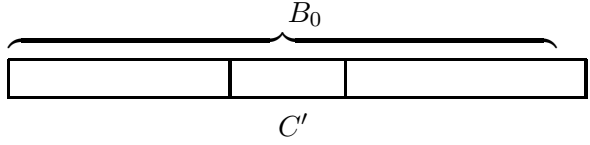
\begin{figure}[ht!]
\begin{center}
\setlength{\unitlength}{0.1\textwidth}
\begin{picture}(6,1)(0,0.2)
\thicklines
\put(0,0.4){\line(1,0){6}}
\put(0,0.8){\line(1,0){6}}

\put(0,0.4){\line(0,1){0.4}}
\put(2.3,0.4){\line(0,1){0.4}}
\put(3.5,0.4){\line(0,1){0.4}}
\put(6,0.4){\line(0,1){0.4}}

\put(0,0.9){$\overbrace{\rule{7.2cm}{0cm}}$}
\put(2.9,1.2){$B_0$}
\put(2.8,0){$C'$}

\end{picture}
\end{center}
\end{figure}

\vspace{0.3 cm}

\noindent $\mathbf{Phase 2:}$ After this is done Black will choose an interval $B_1\subset W_0$ and it is up to White to place $W_1$ inside it. We want White to place $W_1$ inside a cylinder of generation $k_0+k$. It might happen that these generation $k_0+k$-cylinders are so small that White cannot do this right away. But by condition $(ii)$, we know that with $\alpha\leq \alpha_0$ then for every $\beta>0$ there is a strategy for the $(\alpha,\beta)$-game that White can use to place his set inside a generation $k_0+k$ cylinder after a finite number of turns, no matter how Black plays.

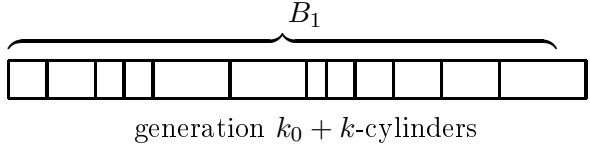
\begin{figure}[ht!]
\begin{center}
\setlength{\unitlength}{0.1\textwidth}
\begin{picture}(6,1)(0,0.2)
\thicklines
\put(0,0.4){\line(1,0){6}}
\put(0,0.8){\line(1,0){6}}

\put(0,0.4){\line(0,1){0.4}}
\put(0.4,0.4){\line(0,1){0.4}}
\put(0.9,0.4){\line(0,1){0.4}}
\put(1.2,0.4){\line(0,1){0.4}}
\put(1.5,0.4){\line(0,1){0.4}}
\put(2.3,0.4){\line(0,1){0.4}}
\put(3.1,0.4){\line(0,1){0.4}}
\put(3.3,0.4){\line(0,1){0.4}}
\put(3.6,0.4){\line(0,1){0.4}}
\put(4,0.4){\line(0,1){0.4}}
\put(4.5,0.4){\line(0,1){0.4}}
\put(5.1,0.4){\line(0,1){0.4}}
\put(6,0.4){\line(0,1){0.4}}

\put(0,0.9){$\overbrace{\rule{7.2cm}{0cm}}$}
\put(2.9,1.2){$B_1$}
\put(1.3,0){generation $k_0+k$-cylinders}

\end{picture}
\end{center}
\end{figure}

If White can place his set $W_1$ inside a cylinder $C_{x_1 \dots x_{k_0+k}}$, then he does. If he cannot, then he uses the following strategy. 

First he places $W_1$ so that it only intersects cylinders $C_{x_1 \dots x_{k_0+k}}$ that are contained in $B_1$. This is possible since if there are subsets of cylinders $C_{x_1 \dots x_{k_0+k}}$ in $B_1$, then these parts together cannot constitute more than $2\alpha |B_1|$, otherwise White would have chosen $W_1$ inside one of them. At his next play, if he can place $W_2$ inside a cylinder $C_{x_1 \dots x_{k_0+k}}$ he does. Otherwise White chooses his set $W_2$ according to a $(\alpha,\alpha \beta^2)$-game strategy that allows him to avoid endpoints of generation $k_0+k$ cylinders after finitely many turns. At his next turn if White could not fit $W_3$ inside a cylinder $C_{x_1 \dots x_{k_0+k}}$ he plays $W_3$ so that he avoids generation $k_0+k$ cylinders that are not contained in $B_3$. As long as he cannot place his set inside a cylinder $C_{x_1 \dots x_{k_0+k}}$ White continues like this, every second turn playing to avoid endpoints of generation $k$ cylinders and the rest of the turns playing to avoid cylinders not contained in the set chosen but Black. Then sooner or later White will be able to place his set inside a cylinder $C_{x_1 \dots x_{k_0+k}}$ and he stops.

Let $j_0$ be the number of the turn at which White could play so that his set $W_j \subset C_{x_1 \dots x_{k_0+k}}$. 
If White needed more than one turn to accomplish this, it means that $C_{x_1 \dots x_{k_0+k}}\subset B_{j_0-2}$. Indeed, at every second play, White makes sure that all cylinders $C_{x_1 \dots x_{k_0+k}}$ that are not fully contained in the set chosen by Black are avoided. We conclude that in this case we have $|C_{x_1 \dots x_{k_0+k}}|\leq |B_{j_0-2}|$.  

\vspace{0.3 cm}

\noindent Now, White has used the turns $0,1,2,\dots ,j_0$ to make the first turn by $\tilde W$ in the sequence building game by choosing between disjoint collections of codings of positions $1, \dots, k_0+k$. He also uses these turns to make sure that the coding of positions $1, \dots, k_0+k$ is fixed after turn number $j_0$. Later on, this fact will allow White to to create the next turn by $\tilde W$. 

After turn $j_0$ by White, Black will choose an interval $B_{j_0+1}\subset W_{j_0}$ and we start creating the next turn by $\tilde W$ in the sequence building game.

\vspace{0.3 cm}

\noindent $\mathbf{Phase 1:}$ Let $k_1$ be the largest generation for which there is a cylinder $C_{x_1 \dots x_{k_1}}$ intersecting $B_{j_0+1}$ such that $|C_{x_1 \dots x_{k_1}}|\geq |B_{j_0+1}|$. Repeating what we did after finding $k_0$, we get that with $\alpha \leq \frac{1}{4}$, White can choose between two disjoint collections of generation $k_1+k$ cylinders. We know that all of these are in the same generation $k_0+k$ cylinder, so White can choose between two disjoint collections of codings at positions $k_0+k+1, \dots, k_1+k$. 

\vspace{0.3 cm}

\noindent $\mathbf{Phase 2:}$ We can then continue as before with $\alpha \leq \alpha_0$, finding a minimal $j_1$ such that $W_{j_1}$ can be placed in a cylinder $|C_{x_1 \dots x_{k_1+k}}|$. Again, if it took more than one turn by White to do this we have $|C_{x_1 \dots x_{k_1+k}}|\leq |B_{j_1-2}|$. 

\vspace{0.3 cm}

\noindent Now, White has used the turns $j_0+1,\dots, j_1$ to make the second turn by $\tilde W$ in the sequence building game by choosing between disjoint collections of codings of positions $k_0+k+1, \dots, k_1+k$ and prepared so that he will be able to make the next turn by $\tilde W$ later on. 

We can continue repeating this procedure for each $i\geq 0$ constructing a turn by $\tilde W$ in which $\tilde W$ gets to choose between disjoint collections of codings at positions $k_i+k+1, \dots, k_{i+1}+k$.

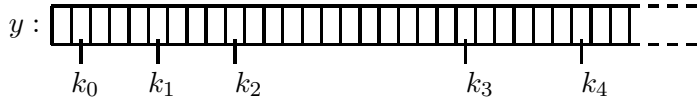
\begin{figure}[ht!]
\begin{center}
\setlength{\unitlength}{0.1\textwidth}
\begin{picture}(6,1)(1,0.2)
\thicklines
\put(0,0.4){\line(1,0){6}}
\put(0,0.8){\line(1,0){6}}

\put(0,0.4){\line(0,1){0.4}}

\put(6,0.4){\line(0,1){0.4}}

\multiput(0,0.4)(0.2,0){30}{\line(0,1){0.4}}

\put(-0.45,0.53){$y:$}
\put(0.2,-0.1){$k_0$}
\put(0.3,0.2){\line(0,1){0.25}}
\put(1,-0.1){$k_1$}
\put(1.1,0.2){\line(0,1){0.25}}
\put(1.9,-0.1){$k_2$}
\put(1.9,0.2){\line(0,1){0.25}}
\put(4.3,-0.1){$k_3$}
\put(4.3,0.2){\line(0,1){0.25}}
\put(5.5,-0.1){$k_4$}
\put(5.5,0.2){\line(0,1){0.25}}

\multiput(6,0.4)(0.2,0){4}{\line(1,0){0.1}}
\multiput(6,0.8)(0.2,0){4}{\line(1,0){0.1}}
\end{picture}
\end{center}
\end{figure}

Next we will show that $k_{i+1}-k_i$ is bounded. We begin by recalling that we had a function $g$ that gave us a speed at which cylinders shrunk in size as the generation increased. We used this function to find a constant $k$ such that when we increased the generation by $k$ the size shrunk by at least a factor $4$. Since the number $k$ originates from potentially very crude estimates it tells us nothing about the size of $k_{i+1}-k_i$. In some cases, it might well happen that $k_{i+1}-k_i=1$ while for example $k=10$. When looking for a uniform bound on $k_{i+1}-k_i$ it will be convenient to consider only the case $k_{i+1}-k_i>k$. Since we are looking for an upper bound, the case $k_{i+1}-k_i\leq k$ is uninteresting.

We start at phase $1$ when constructing turn number $i$ for $\tilde W$ in the sequence building game. First Black plays by choosing a set $B_1$, then White chooses between two disjoint collections of generation $k_i+k$ cylinders. Then phase $2$ starts as Black plays again. Assume now that White is able to place his set inside a generation $k_i+k$ cylinder at his first turn in phase $2$. This ends phase $2$ and means the end of turn number $i$ for $\tilde W$ in the sequence building game.

After this, it is time to construct turn number $i+1$ by $\tilde W$. Black starts phase $1$ by choosing a set $B_2$. Then we find the maximal generation $k_{i+1}$ such that $B_2$ intersects a cylinder $C_{x_1 \dots x_{k_i+1}}$ such that $|C_{x_1 \dots x_{k_i+1}}|\geq |B_2|$. Since $|B_2|=(\alpha\beta)^2 |B_1|$ we get
\begin{align*}
|B_{2}| \leq &|C_{x_1\dots x_{k_{i+1}}}| \leq g(k_{i+1}-k_i-k)|C_{k_1 \dots x_{k_i+k}}|\\
< &g(k_{i+1}-k_i-k)|B_1|=\frac{g(k_{i+1}-k_i-k)|B_2|}{(\alpha \beta )^2}
\end{align*}
so $g(k_{i+1}-k_i-k) <(\alpha \beta)^2$. Since $g(n)\to 0$ as $n\to \infty$ this puts a bound on $k_{i+1}-k_i$.

Assume instead that when constructing turn number $i$ for $\tilde W$ in the sequence building game, White needed more than one turn in phase $2$, to place his set inside a generation $k_i+k$ cylinder. We recall that if $W_j$ is the last set chosen by White in this phase, then $|C_{x_1 \dots x_{k_i}}|< |B_{j-2}|$. After this, it is time to construct turn number $i+1$ by $\tilde W$. Black starts phase $1$ by choosing a set $B_{j+1}$. Then we find the maximal generation $k_{i+1}$ such that $B_{j+1}$ intersects a cylinder $C_{x_1 \dots x_{k_{i+1}}}$ such that $|C_{x_1 \dots x_{k_{i+1}}}|\geq |B_{j+1}|$. We get
\begin{align*}
|B_{j+1}| \leq &|C_{x_1\dots x_{k_{i+1}}}| \leq g(k_{i+1}-k_i-k) |C_{x_1 \dots x_{k_{i}+k}}| \\
< & g(k_{i+1}-k_i-k)|B_{j-2}|=\frac{g(k_{i+1}-k_i-k)|B_{j+1}|}{(\alpha \beta)^3},
\end{align*}
so $g(k_{i+1}-k_i-k) <(\alpha \beta)^3$. Since $g(n)\to 0$ as $n\to \infty$ this puts a bound on $k_{i+1}-k_i$.

What we have proven this far is that if we choose $\alpha\leq \min\{\alpha_0,\frac{1}{4}\}$, then the sequence $k_i$ has a maximal distance between its elements. This implies that if the block size $n$ is large enough, then in the following picture

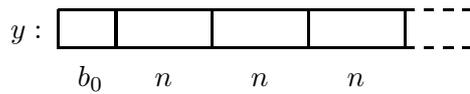
\begin{figure}[ht!]
\begin{center}
\setlength{\unitlength}{0.1\textwidth}
\begin{picture}(6,1)(1,0)
\thicklines
\put(0,0.4){\line(1,0){3.6}}
\put(0,0.8){\line(1,0){3.6}}
\multiput(3.6,0.4)(0.2,0){4}{\line(1,0){0.1}}
\multiput(3.6,0.8)(0.2,0){4}{\line(1,0){0.1}}

\put(0,0.4){\line(0,1){0.4}}
\put(0.6,0.4){\line(0,1){0.4}}
\put(1.6,0.4){\line(0,1){0.4}}
\put(2.6,0.4){\line(0,1){0.4}}
\put(3.6,0.4){\line(0,1){0.4}}

\put(-0.5,0.5){$y:$}
\put(0.2,0){$b_0$}
\put(1,0){$n$}
\put(2,0){$n$}
\put(3,0){$n$}

\end{picture}
\end{center}
\end{figure}

\noindent there is at least one $k_i$ in each $n$-block. Increasing $n$ we can clearly make sure that there are at least $cn$ different $k_i$ in each $n$-block for some $c>0$. This implies that if we play the $(\alpha,\beta)$-game in $[0,1]$ with $\alpha\leq \min\{\alpha_0, \frac{1}{4}\}$, then White can use a strategy that transforms the game into the sequence building game. By Proposition \ref{digitgame} the player $\tilde W$ can make sure that we get a number in $\Big\{\, z\in [0,1): x\notin \overline{\cup_{n=1}^\infty f^n(z)}\, \,\Big\}$ for any given $x\in [0,1)$ with well-defined expansion, by choosing the block size $n$ in the sequence building game. Since this can be done for any $\alpha\leq \min\{\alpha_0, \frac{1}{4}\}$ and any $\beta>0$ we conclude that $\Big\{\, z\in [0,1): x\notin \overline{\cup_{n=1}^\infty f^n(z)}\, \, \Big\}$ is $\alpha$-winning for all $x\in [0,1)$ with well-defined expansion, if $\alpha \leq \min\{\alpha_0, \frac{1}{4}\}$. This proves the theorem.

\subsection{Proof of Theorem \ref{SFTunitinterval}}

Since the symbol $\beta$ is already used in the $(\alpha,\beta)$-game we will use $b$ instead of $\beta$ to denote the base in the $\beta$-shift.

The method used to prove Theorem \ref{fullshiftinunitinterval} works in this case as well, but now we do not have to worry about countable alphabets and points without well-defined expansions. Since $S_b$ is of finite type there is a constant $C_b$ such that 
\[
C_b^{-1}<\frac{C_{x_0\dots x_{n-1}}}{b^{n}}<C_b
\]
for all $n$ and all $(x_i)_{i=0}^\infty \in S_b$. This implies that 
\[
\frac{|C_{x_0\dots x_{n+m}}|}{|C_{x_0\dots x_{n}}|} \leq \frac{C_b^2}{b^m}
\]
for all $m,n$ and all $(x_i)_{i=0}^\infty \in S_b$. Thus we can let $\frac{C_b^2}{b^m}$ play the role of $g(m)$ from the proof of Theorem \ref{fullshiftinunitinterval}.

We will now briefly describe how White plays in the $(\alpha,\beta)$-game to construct turn number $i$ in the sequence building game. It all begins as usual with the player Black choosing a set $B_i$. 

\vspace{0.3 cm}

\noindent $\mathbf{Phase 1:}$ We do as in proof of Theorem \ref{fullshiftinunitinterval}. We find a minimal $k_i$. Then we choose $k$ large enough so that by placing $W_i$ White can choose between two disjoint collections of generation $k_i+k$ cylinders.  For example, $k\geq 1+\frac{4C_b^2}{\log b}$ will be enough. 

\vspace{0.3 cm}

\noindent $\mathbf{Phase 2:}$ We do as in the proof of Theorem \ref{fullshiftinunitinterval}. We let White alternate between avoiding cylinders not contained in the sets chosen by Black and avoiding endpoints of generation $k_i+k$ cylinders until White can place his set in a generation $k_i+k$ cylinder.

\vspace{0.3 cm}

\noindent Just as in the proof of Theorem \ref{fullshiftinunitinterval} we conclude that with this sequence of turns, White is able to choose between disjoint collections of codings of positions $k_{i-1}+k+1, \dots k_i+k$. We then do as in proof of Theorem \ref{fullshiftinunitinterval} to show that $k_{i+1}-k_i$ is bounded. We then apply Proposition \ref{digitgame} to conclude that $\Big\{\, z\in [0,1): x\notin \overline{\cup_{n=1}^\infty f_b^n(z)}\, \, \Big\}$ is $\alpha$-winning for all $x\in [0,1]$ and all $\alpha \leq \frac{1}{4}$. This proves the theorem.

\subsection{A note on the {$\boldsymbol{(\alpha,\beta)}$-game}}
We note that in the proofs of Theorem \ref{fullshiftinunitinterval} and Theorem \ref{SFTunitinterval}, the strategies we describe for White use the fact that Black can not zoom in more than a fixed factor $\gamma$ at each turn in each given game. It would not matter at all for the strategies if Black was allowed at each turn to choose $\gamma \in [\gamma_0,1] $ for some fixed $\gamma_0$. If we also allow White to choose $\alpha \in [\alpha_0,1] $, White can still use the same strategy. This leads us to consider the following modification of the $(\alpha,\beta)$-game.

Let $\alpha_0,\gamma_0 \in (0,1)$ be fixed.
    \begin{itemize}
      \item{\em In the initial step} Black chooses any closed interval $B_0$, and then White chooses an $\alpha \in [\alpha_0, 1]$ and a closed interval $W_0 \subset B_0$ such that $|W_0| = \alpha |B_0|$.
      \item{\em Then the following step is repeated.} At step $k$ Black choses $\gamma\in [\gamma_0,1]$ and a closed interval $B_{k} \subset W_{k-1}$ such that $|B_{k}| = \gamma |W_{k-1}|$. Then White chooses a new $\alpha \in [\alpha_0, 1]$ and a closed interval $W_{k} \subset B_k$ such that $|W_k| = \alpha |B_k|$.
    \end{itemize}

The following observation now follows.
\begin{rem}
In Theorems \ref{fullshiftinunitinterval} and \ref{SFTunitinterval} with corollaries, the $(\alpha,\beta)$-game can be replaced by the modified $(\alpha,\beta)$-game described in this section.
\end{rem}

\end{document}